\documentclass[11pt,reqno]{amsart}
\usepackage[margin = 1 in]{geometry}
\usepackage[frenchmath,defaultmathsizes]{mathastext}
\usepackage{
  hyperref,
  amsmath,
  amssymb,
  amsthm,
  thmtools,
  microtype,
  mathrsfs,
  enumitem,
  stmaryrd,
  diagbox
}
\usepackage[table]{xcolor}
\usepackage{tikz, tikz-cd}
\usepackage{cleveref}
\usepackage{tikz-cd}
\usepackage{url}
\usepackage{verbatim}
\usepackage{pgfplots} 
\pgfplotsset{width=8cm,compat=1.9} 
\usepackage{subcaption}

\usepackage{graphicx}

\usepackage{eucal}

\theoremstyle{plain}
\newtheorem{theorem}{Theorem}[section]
\newtheorem{proposition}[theorem]{Proposition}
\newtheorem{lemma}[theorem]{Lemma}
\newtheorem{conjecture}[theorem]{Conjecture} 
\newtheorem{corollary}[theorem]{Corollary}
\theoremstyle{definition}
\newtheorem{definition}[theorem]{Definition}

\newtheorem{example}[theorem]{Example}
\newtheorem{exercise}[theorem]{Exercise}

\newtheorem{question}[theorem]{Question}
\theoremstyle{remark}

\numberwithin{equation}{section}

\newcommand{\cO}{{\mathscr O}}
\renewcommand{\to}{{\longrightarrow}}

\ifcsname theorem\endcsname{}\else\declaretheorem[parent=section]{theorem}\fi
\ifcsname corollary\endcsname{}\else\declaretheorem[sibling=theorem]{corollary}\fi
\ifcsname lemma\endcsname{}\else\declaretheorem[sibling=theorem]{lemma}\fi
\ifcsname proposition\endcsname{}\else\declaretheorem[sibling=theorem]{proposition}\fi
\ifcsname conjecture\endcsname{}\else\fi
\ifcsname problem\endcsname{}\else\fi
\ifcsname question\endcsname{}\else\declaretheorem[sibling=theorem]{question}\fi
\ifcsname definition\endcsname{}\else\declaretheorem[sibling=theorem, style=definition]{definition}\fi
\ifcsname exercise\endcsname{}\else\fi
\ifcsname example\endcsname{}\else\fi
\ifcsname remark\endcsname{}\declaretheorem[sibling=theorem, style=remark]{remark}\fi

\providecommand{\PGL}{\operatorname{PGL}}

\providecommand{\Hom}{\operatorname{Hom}}

\numberwithin{equation}{section}

\newcommand{\OO}{\mathscr{O}}
\newcommand{\PP}{\mathbb{P}}

\DeclareMathOperator{\Mor}{Mor}

\newcommand{\smvee}{\raise0.5ex\hbox{$\scriptscriptstyle\vee$}}

\newcommand{\cM}{\mathcal{M}}
\newcommand{\cU}{\mathcal{U}}
\newcommand{\CC}{\mathbb{C}}

\newcommand{\Spec}{{\text{\rm Spec}\,}}

\newcommand{\git}{\sslash}
\newcommand{\up}{\textrm{upper}}
\newcommand{\low}{\textrm{lower}}

\title{Moduli of linear slices of high degree smooth hypersurfaces}
\author{Anand Patel, Eric Riedl, Dennis Tseng}

\begin{document}

\maketitle

\begin{abstract}
We study the variation of linear sections of hypersurfaces in $\PP^n$. We completely classify all plane curves, necessarily singular, whose line sections do not vary maximally in moduli. In higher dimensions, we prove that the family of hyperplane sections of any smooth degree $d$ hypersurface in $\PP^n$ varies  maximally for $d \geq n+3$. In the process, we generalize the classical Grauert-Mulich theorem about lines in projective space, both to $k$-planes in projective space and to free rational curves on arbitrary varieties.
\end{abstract}

\section{Introduction}
A fundamental technique for studying a degree $d$ complex hypersurface $X$ in projective space $\PP^n$ is to intersect it with hyperplanes.  The family of varieties thus obtained can be represented by a map to moduli:
\begin{align*}
    \phi: \mathbb{P}^{n*}&\dashrightarrow \PP H^0(\OO_{\PP^{n-1}}(d))\git SL_{n} \\
    [\Lambda] &\mapsto [\Lambda\cap X]
\end{align*}
 Basic properties of $\phi$ are still not understood, even under regularity assumptions on $X$.  Take, for instance, the problem of determining the dimension of its image.  If $X$ is assumed to be {\sl general}, then $\phi$ can directly be shown to have maximal rank, i.e. its image is as large as possible, as done in \cite{opstallveliche}. However, once we assume $X$ is an {\sl arbitrary} hypersurface, the story becomes more complicated, with several authors studying special cases in the last few decades.  Even in the case of a reduced plane curve $X$, showing maximal variation is not a trivial task. Thirty years ago, while studying $\PGL_3$-orbits of plane curves, Aluffi and Faber \cite[Proposition 4.2]{AF93P} cleverly exploited the classical Pl\"ucker formulas to prove that {\sl smooth} plane curves of degree at least 5 always have maximum variation of linear sections.   However, if the curve $X$ is singular, then $\phi$ can fail to have maximal rank, and Aluffi and Faber were not able to completely analyze this case.  
 
    Quite generally, if the dimension of the projective automorphism group of $X$ is larger than expected (e.g. if $X$ is a cone), then linear slices must fail to vary maximally in moduli.  Outside this class of hypersurfaces, we are unaware of any other examples where $\phi$ fails to have maximal rank, so we pose the following question: 
\begin{question}
\label{question:Automorphisms}  If $\phi$ fails to have maximal rank, must $X$ have a positive-dimenstional projective automorphism group?
\end{question}

Our first result is to answer \Cref{question:Automorphisms} affirmatively when $X \subset \PP^{2}$ is a plane curve:
\begin{theorem}
\label{thm:curveintro}
If $X\subset \PP^2_{\CC}$ is an arbitrary plane curve and if $\phi$ fails to have maximal rank, then $X$ has infinitely many projective automorphisms. 
\end{theorem}
Given \Cref{thm:curveintro}, we see that of all curves where $\phi$ fails to have maximal rank have stabilizer containing $\mathbb{G}_m$ or $\mathbb{G}_a$, where typical examples are depicted in \Cref{plots}.
\begin{figure}[!h]
\centering
\begin{subfigure}{.475\linewidth}
\centering
\begin{tikzpicture}
\begin{axis}[
    trig format plots=rad,
    axis equal,
    hide axis
]

\addplot [domain=-.05:1, samples=50, black, dashed] ({0}, {x});
\addplot [domain=-.05:1, samples=50, black, dashed] ({1}, {x});

\addplot [domain=-.7:5, samples=300, black] ({1/(1+x^5)}, {x^2/(1+x^5)});
\addplot [domain=-.7:0, samples=50, black] ({1/(1+x^(-5))}, {-x^(-2)/(1+x^(-5))});
\addplot [domain=-.7:5, samples=300, black] ({1/(1+x^5)}, {2*x^2/(1+x^5)});
\addplot [domain=-.7:0, samples=50, black] ({1/(1+x^(-5))}, {-2*x^(-2)/(1+x^(-5))});

\addplot [domain=-.7:5, samples=300, black] ({1/(1+x^5)}, {1.5*x^2/(1+x^5)});
\addplot [domain=-.7:0, samples=50, black] ({1/(1+x^(-5))}, {-1.5*x^(-2)/(1+x^(-5))});

\addplot [domain=-.7:5, samples=100, black] ({1/(1+x^5)}, {1.25*x^2/(1+x^5)});
\addplot [domain=-.7:0, samples=50, black] ({1/(1+x^(-5))}, {-1.25*x^(-2)/(1+x^(-5))});

\addplot [domain=-.7:5, samples=100, black] ({1/(1+x^5)}, {1.75*x^2/(1+x^5)});
\addplot [domain=-.7:0, samples=50, black] ({1/(1+x^(-5))}, {-1.75*x^(-2)/(1+x^(-5))});
\end{axis}
\end{tikzpicture}
\caption{Union of orbits under $\mathbb{G}_m$ action}
\end{subfigure}
\hfill%
\begin{subfigure}{.475\linewidth}
\begin{tikzpicture}
\begin{axis}[
    axis equal,
    hide axis
]

\addplot [domain=0:360, samples=60,black] ({sin(x)+cos(x)},{cos(x)});
\addplot [domain=0:360, samples=60,black] ({.5+.5*(sin(x)+cos(x))},{.5*(cos(x))});
\addplot [domain=0:360, samples=60,black] ({.25+.75*(sin(x)+cos(x))},{.75*(cos(x))});
\addplot [domain=0:360, samples=60,black] ({.75+.25*(sin(x)+cos(x))},{.25*(cos(x))});
\end{axis}
\end{tikzpicture}
\caption{Union of orbits under $\mathbb{G}_a$ action (quadritangent conics)}
\end{subfigure}
\caption{Singular curves for which $\phi$ fails to have maximal rank}
\label{plots}
\end{figure}
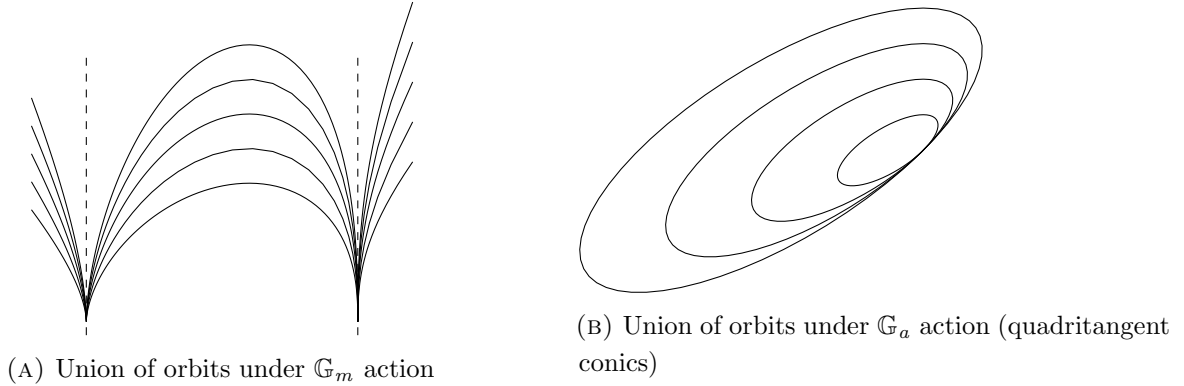
 
The map $\phi$ is even more difficult to understand for larger-dimensional hypersurfaces -- we restrict our attention primarily to smooth hypersurfaces in this paper. Beauville in \cite{beauville} investigated the case where $\phi$ is a {\sl constant map} and classified the smooth hypersurfaces $X$ for which the family of hyperplane sections has constant moduli.   This phenomenon happens only for very special hypersurfaces in positive characteristic. In contrast, we prove:

 \begin{theorem}
\label{thm:hypsliceintro}
If $X\subset \PP^n_{\CC}$ is a smooth hypersurface of degree $d\geq n+3$, then $\phi$ has maximal rank. 
\end{theorem}

We can also intersect a hypersurface $X$ with $k$-planes for smaller $k$, obtaining natural analogues
\begin{align*}
    \phi_k: \mathbb{G}(k,n)&\dashrightarrow \PP H^0(\OO_{\PP^{k}}(d))\git SL_{k+1} \\
    [\Lambda] &\mapsto [\Lambda\cap X]
\end{align*}
and ask similar questions about $\phi_k$.   Harris, Mazur, and Pandharipande \cite{HMP95}, and then later Starr \cite{starr}, studied the situation where $\phi_k$ is expected to be dominant, relating the problem of establishing dominance to the question of unirationality of low degree hypersurfaces. When $\phi_{k}$ is expected to be generically finite and dominant, the problem of establishing its degree has also appeared in the literature.  In this direction, see \cite{CL08,FML,LPT19}.
 
We are able to generalize \autoref{thm:hypsliceintro}, and prove that $\phi_{k}$ has maximal rank under some restrictions on $k$:

\begin{theorem}
\label{thm:hypsliceintro2}
If $X\subset \PP_{\CC}^n$ is a smooth hypersurface of degree $d$, then $\phi_{k}$ has maximal rank assuming $d\geq n+3$ and $k\geq \frac{2}{3}n$. 
\end{theorem}
For $k<\frac{2}{3}n$, we obtain similar statements, but with $d$ forced to be larger (see \Cref{thm:knboundslice} and \Cref{thm:smallkslice}). 

Broadening the topic even further, we can intersect $X$ with other types of varieties, for example rational curves of degree $e$.  In this way, we obtain a map from the variety of degree $e$ rational curves in $\PP^{n}$ to the moduli space of $ed$ points on $\PP^{1}$.  Our methods also provide results in this context -- see \Cref{classification} and \Cref{theorem:degEcurvesLargen}.

\subsection{Methods}
The log tangent sheaf $T_{\PP^n}(-\log X)$ and the Grauert-Mulich theorem play key roles in our approach. We identify the tangent space of the fiber of $\phi$ at a point $[\Lambda]$ with sections of the log tangent sheaf $T_{\PP^n}(-\log X)$ restricted to $\Lambda$. Then, we adapt the argument in the usual Grauert-Mulich theorem \cite[Theorem 2.1.4]{OSS80} to produce sections or subsheaves of the log tangent sheaf $T_{\PP^n}(-\log X)$. In the plane curve case, this forces $X$ to be an integral curve for a vector field on $\PP^2$,  leading to the classification in \Cref{classification}. In the higher dimensional case, we appeal to a result of Guenancia regarding the semistability of $T_{\PP^n}(-\log X)$, when $(\PP^n, X)$ is a log-canonical pair and $d\geq n+2$. In particular, all our results in this case actually hold when $(\PP^n, X)$ is a log-canonical pair, not only when $X$ is smooth. 

Our methods will produce results in other contexts, for example if we replace $\PP^n$ with a homogeneous space $G/P$. 

\subsection{Acknowledgements}
The authors would like to thank Paolo Aluffi, Izzet Coskun, Carel Faber and Joe Harris for helpful conversations.

\section{Preliminaries}
In this section, we introduce conventions and basic definitions.
\subsection{Notation and Conventions} 
We will work over the complex numbers.  We identify vector bundles with locally free sheaves throughout, and all our sheaves are coherent.  A {\sl sub-bundle} of a vector bundle $V$ is a locally free subsheaf $W \subset V$ such that $V/W$ is also locally-free. For us, a variety is an integral scheme of finite type.  If $F$ is a coherent sheaf on a scheme $X$, we denote by $\operatorname{ev}: H^{0}(X,F) \otimes \OO_{X} \to F$ the natural evaluation map.

We denote by $\Mor_{e}(\PP^{k},\PP^{n})$ the variety parametrizing morphisms $f: \PP^{k} \to \PP^{n}$ with $f^{*}\OO(1) = \OO(e)$.  Explicitly, $\Mor_{e}(\PP^{k},\PP^{n})$ is a Zariski open subset of $\PP(H^{0}(\OO_{\PP^{k}}(e))^{\oplus n+1})$ parametrizing tuples $(A_{0}, \dots, A_{n})$ of homogeneous degree $e$ forms on $\PP^{k}$ which do not vanish simultaneously anywhere on $\PP^{k}$. More generally, $\Mor(X,Y)$ denotes the (not finite-type) scheme parameterizing morphisms from the scheme $X$ to the scheme $Y$. 

Given a torsion-free sheaf $E$ on a projective variety $X$, we let its slope $\mu(E)$ denote the ratio $\frac{\deg(E)}{\operatorname{rank}(E)}$, where $\deg(E) = \int_{X}c_1(E)\OO_X(1)^{\dim(X)-1}$. We call $E$ semistable (respectively stable) if there is no proper subsheaf $F$ with $\mu(F)>\mu(E)$ (respectively $\mu(F)\geq \mu(E)$). In general, the \emph{Harder-Narasimhan} filtration of $E$ is
\begin{align*}
    0=E_0\subsetneq E_1\subsetneq E_2\cdots \subsetneq E_a=E,
\end{align*}
where the subquotients $E_1/E_0,E_2/E_1,\ldots,E_a/E_{a-1}$ are semistable and have strictly decreasing slopes. 

\subsection{The map to moduli $\Phi$}\label{subsection:maptomoduli}
Suppose $X \in \PP^{n}$ is a degree $d$ hypersurface.  After fixing integers $e \geq 1, k \leq n-1$, we get the induced {\sl map to moduli}
\begin{align*}
    \Phi: \operatorname{Mor}_e(\PP^k,\PP^n)&\dashrightarrow \PP H^0(\PP^k,\OO_{\PP^k}(de))\\
    \iota &\mapsto [\iota^{-1}(X)].
\end{align*}
We say that $\Phi$ has maximal rank  if the dimension of its image is $\max\{\dim(\operatorname{Mor}_e(\PP^k,\PP^n)),\PP H^0(\PP^k,\OO_{\PP^k}(de))\}$. Equivalently, since we are working over $\mathbb{C}$, the derivative of $\Phi$ at a general point has maximum rank. 

Though our methods give results for all $e,k$, we are primarily interested in the cases where $e=1$ or $k=1$.  Therefore, we have only stated our results in these cases. In the case $e=1$, $\Phi$ having maximal rank is equivalent to the map
\begin{align*}
    \mathbb{G}(k,n)&\dashrightarrow \PP H^0(\PP^k,\OO_{\PP^k}(d))\git SL_k\\
    [\Lambda]&\mapsto [\Lambda\cap X]
\end{align*}
having maximal rank, assuming the general $k$-plane slice of $X$ has no infinitesimal automorphisms. Whenever our results apply, this condition will always be satisfied. 
\subsection{Log tangent sheaves}
We now introduce the main tool of the paper. We suspect \Cref{logtangent} is well-known to experts but include a proof for want of a suitable reference. Everything in this section should work for a reduced divisor in an arbitrary smooth ambient variety, but we will focus on the case that the ambient variety is projective space in this paper. 

Let $D\subset \mathbb{P}^n$ be a reduced hypersurface. Viewing $D$ as a divisor in the smooth ambient variety $\mathbb{P}^n$, we get the log tangent sheaf $T_{\mathbb{P}^n}(-\log D)$, which sits inside the exact sequence
\begin{align*}
    0\to T_{\mathbb{P}^n}(-\log D)\to T_{\mathbb{P}^n}\to \mathscr{O}_D(D)\to \mathscr{O}_{D_{{\rm sing}}}(D)\to 0,
\end{align*}
where $D_{{\rm sing}}$ is the singular subscheme cut out of $\mathbb{P}^n$ by the equation for $D$ and its partials. In terms of background, we  only assume what is covered in \cite[2.1.2]{Xia93}, but see \cite{S80} for the original reference. One can check that $T_{\mathbb{P}^n}(-\log D)$ is a vector bundle when $D$ is smooth using local coordinates; in general $T_{\mathbb{P}^n}(-\log D)$ is a reflexive sheaf. 

Informally, local sections of $T_{\mathbb{P}^n}(-\log D)$ represent local vector fields which are tangent to $D$. This can be seen explicitly by noting that the map $T_{\mathbb{P}^n}\to \cO_D(D)$ in the exact sequence above is given by $\theta\mapsto \theta(f)$, where $\theta$ is a vector field and $f$ is the (local) equation for $D$. If we identify $\cO_D(D)$ with $N_{D/\mathbb{P}^n}$, the map $T_{\mathbb{P}^n}\to \cO_D(D)$ is also $T_{\mathbb{P}^n}\to T_{\mathbb{P}^n}|_{D}\to N_{D/\mathbb{P}^n}$.

Let $\mathbb{P}^k\xrightarrow{\iota} \mathbb{P}^n$ be a map defined by degree $e$ homogeneous forms, and suppose $Z \subset \mathbb{P}^{n}$ is a subscheme.  We say an infinitesimal deformation $\iota_{\epsilon} : \mathbb{P}^{k}\times\Spec \mathbb{C}[\epsilon]/(\epsilon^{2}) \to \mathbb{P}^{n}$ {\sl preserves $\iota^{-1}(Z)$} if $\iota_{\epsilon}^{-1}(Z) \subset \mathbb{P}^{k}\times\Spec \mathbb{C}[\epsilon]/(\epsilon^{2})$ is the trivial deformation $\iota^{-1}(Z)\times \Spec \mathbb{C}[\epsilon]/(\epsilon^{2})$.
The point of this section is to prove the following lemma.
\begin{lemma}
\label{logtangent}
Let $\mathbb{P}^k\xrightarrow{\iota} \mathbb{P}^n$ be a map defined by degree $e$ homogeneous forms whose image is not contained in $D$. Global sections of $\iota^{*}T_{\mathbb{P}^n}(-\log D)$ correspond to deformations of the map $\iota$ preserving the hypersurface $\iota^{-1}(D)$. 
\end{lemma}

\begin{proof}
First, sections of $\iota^{*}T_{\mathbb{P}^n}$ correspond to deformations of $\iota$. More explicitly, $\iota$ is defined by an $n+1$ tuple of degree $e$ forms in $k+1$ variables $A_0(s_0,\ldots,s_k),\ldots,A_n(s_0,\ldots,s_k)$ sending $[s_0:\cdots:s_k]$ to $[A_0(s_0,\ldots,s_k):\cdots:A_n(s_0,\ldots,s_k)]$. 

A deformation $\iota_{\epsilon}$ of $\iota$ is given by another $n+1$ tuple of degree $e$ forms in $k+1$ variables $B_0(s_0,\ldots,s_k),\ldots,B_n(s_0,\ldots,s_k)$. Explicitly, as a map from $\operatorname{Spec}(\mathbb{C}[\epsilon]/(\epsilon^2))\times \mathbb{P}^k\to \mathbb{P}^n$ this is given in coordinates by $\epsilon,[s_0,\ldots,s_k]$ mapping to $[A_0(s_0,\ldots,s_k)+\epsilon B_0(s_0,\ldots,s_k):\cdots:A_n(s_0,\ldots,s_k)+\epsilon B_n(s_0,\ldots,s_k)]$. The vector space of deformations is given by the quotient space of $n+1$-tuples of degree $e$ forms $(B_0,\ldots,B_{n})$ modulo the 1-dimensional vector space generated by $(A_0,\ldots,A_{n})$. 

Let $D$ be defined by $F=0$ where $F$ is a reduced homogenous form in $n+1$ variables. If we pull back the form $F$ under the deformed map  $\iota_{\epsilon}$, we obtain
\begin{align*}
    F(A_0(s_0,\ldots,s_k)+\epsilon B_0(s_0,\ldots,s_k),\ldots,A_n(s_0,\ldots,s_k)+\epsilon B_n(s_0,\ldots,s_k))&=\\
    F(A_0(s_0,\ldots,s_k),\ldots,A_n(s_0,\ldots,s_k))+&\\\epsilon \sum_{i=0}^{n}B_i(s_0,\ldots,s_k) \cdot \partial_iF (A_0(s_0,\ldots,s_k),\ldots,A_n(s_0,\ldots,s_k))&. 
\end{align*}
Therefore, deformations $\iota_{\epsilon}$ that preserve $\iota^{-1}(D)$ correspond to choices of $B_0,\ldots,B_n$ such that
\begin{align}
\label{multiple}
    \sum_{i=0}^{n}B_i(s_0,\ldots,s_k)\cdot \partial_iF (A_0(s_0,\ldots,s_k),\ldots,A_n(s_0,\ldots,s_k))
\end{align}
is a scalar multiple of $F$. 

Now, we wish to realize this latter condition as producing sections of the pulled back log tangent sheaf. First, the sections of the pulled back tangent sheaf $\iota^{*}T_{\mathbb{P}^n}$ can be computed via the Euler sequence
\begin{align*}
    0\to \mathscr{O}_{\mathbb{P}^k}\to \mathscr{O}_{\mathbb{P}^k}(e)^{n+1}\to \iota^{*}T_{\mathbb{P}^n}\to 0
\end{align*}
to be the quotient space of $n+1$-tuples of linear forms $(B_0,\ldots,B_{n})$ modulo the 1-dimensional vector space generated by $(A_0,\ldots,A_{n})$. 

The restricted vector field corresponding to $(B_0,\ldots,B_{n})$ is $\sum_{i=0}^{n}B_i\frac{\partial}{\partial x_i}$. Recall that $T_{\mathbb{P}^n}(-\log D)$ is the kernel of the map $T_{\mathbb{P}^n}\to \mathscr{O}_D(D)$ sending a vector field $\theta:=\sum_{i=0}^{n}B_i\frac{\partial}{\partial x_i}$ to $\theta(F)$. 

In other words, the defining equation is 
\begin{align*}
    \sum_{i=0}^{n}B_i\frac{\partial}{\partial x_i}F \equiv 0 \pmod{F}.
\end{align*}
Pulling back this under $\iota$ yields exactly \eqref{multiple}. 
\end{proof}

\section{Grauert-Mulich}
\label{sec:Grauert-Mulich}
The goal of this section is to generalize the classical Grauert-Mulich theorem \cite[Theorem 2.1.4]{OSS80} in two directions:

\begin{proposition}
\label{thm-GrauertMulich}
Let $Z\subset \mathbb{P}^N$ be a smooth projective variety and $f: \mathbb{P}^1\to Z$ be a general map in an open subset $\cM$ of $\operatorname{Mor}(\mathbb{P}^1, Z)$ such that $\PP^1\times \cM\to Z$ has connected fibers. Suppose $f^{*}T_Z$ is globally generated. 

Let $E$ be a torsion free sheaf on $Z$ and write $f^{*}E$ as $\bigoplus_{i=1}^b \cO(a_i)$ with $a_1\geq \cdots \geq a_b$. If $a_j>a_{j+1}+1$ for some $i$, then $E$ has a subsheaf of rank $i$ and degree $\frac{1}{\deg(f)}\sum_{i=1}^ja_i$. In particular, if $E$ is semistable, then the bundle $f^{*}E$ can be written as $\bigoplus_i \cO(a_i)$ with $|a_i - a_{i+1}| \leq 1$.
\end{proposition}

For applications to slicing by $k$-planes, we will use \Cref{lem:slopesOfRestrictionsTokPlanes}.
\begin{definition}
Given a torsion free sheaf $E$ on a smooth projective variety, let $\mu^{\max}(E)$ denote the maximum slope of a nontrivial subsheaf of $E$. It is also the slope of the first subsheaf appearing in its Harder-Narasimhan filtration. Similarly let $\mu^{\min}(E)$ denote the minimum slope of a nontrivial quotient of $E$. It is also the slope of the quotient of the last two subsheaves appearing in its Harder-Narasimhan filtration.
\end{definition}

\begin{proposition}
\label{lem:slopesOfRestrictionsTokPlanes}
Let $E$ be a torsion free sheaf on $\mathbb{P}^n$. Let $\Lambda$ be a general $k$-plane in $\mathbb{P}^n$. Let $S\subsetneq E|_{\Lambda}$ be a sheaf appearing in the Harder-Narasimhan filtration of $E|_{\Lambda}$ and suppose 
\begin{align*}
\mu^{\min}(S)-\frac{1}{k}>\mu^{\max}(E|_{\Lambda}/S),
\end{align*}
then $E$ is not semistable.
\end{proposition}

The proofs of \Cref{thm-GrauertMulich} and \Cref{lem:slopesOfRestrictionsTokPlanes} are very similar in spirit to standard proofs of Grauert-Mulich, such as the one found in \cite{OSS80}. The argument relies crucially on the following Lemma.

\begin{lemma}[{Descente Lemma from \cite[Lemma 2.1.2]{OSS80}}]
\label{lem-descent}
Let $Y$ and $Z$ be smooth varieties and $\pi: Y\to Z$ be a surjective smooth morphism with connected fibers. Let $E$ be a vector bundle on $Z$ such that $\pi^{*}E$ has a vector subbundle $S$ with quotient vector bundle $Q$. Then, if $\Hom(T_{Y/Z},\Hom(S,Q))=0$, then $S$ is the pullback of a subbundle of $E$ on $Z$. 
\end{lemma}

The key technical lemma of this section is \Cref{integrate}, whose proof will use the following simple fact in \Cref{subquots}.

\begin{lemma}
\label{subquots}
Let $Y$ be a variety and $E$ and $F$ be two sheaves on $Y$. Suppose all the semistable subquotients in the Harder-Narasimhan filtration of $E$ have greater slope than the corresponding semistable subquotients of $F$. Then, $\Hom(E,F)=0$.
\end{lemma}

\begin{proof}
Let $0=E_0\subset E_1\subset \cdots\subset E_a=E$ be the Harder-Narasimhan filtration for $E$ and $0=F_0\subset F_1\subset \cdots\subset F_b=F$ be the Harder-Narasimhan filtration for $F$. Consider a map $\phi: E\to F$. We show $\phi=0$. 

First, $\phi$ induces a map $E_1\to F_b/F_{b-1}$, which is zero since the source is semistable and has slope greater than the target, which is also semistable. Therefore, $\phi$ induces a map $E_1\to F_{b-1}/F_{b-2}$, which again is zero for the same reason. Continuing this, we find the map $E_1\to F$ is zero. 

Then, we consider the induced map $E_2/E_1\to F_{b}/F_{b-1}$ and repeat the argument above to find $E_2/E_1\to F$ must be the zero map. Continuing this for $E_3/E_2$ and so on shows that the map $\phi$ is zero. 
\end{proof}

\begin{lemma}
\label{integrate}
Let $Z$ be a smooth projective variety and let $\cU \to \cM$ be a smooth family of projective varieties with a smooth surjective map $\pi:\cU \to Z$ having connected fibers. Let $E$ be a torsion free sheaf on $Z$ and let $\cU_p$ be a general fiber of $\cU\to \cM$. Let $S$ be a subsheaf of $\pi^{*}E|_{\cU_p}$ appearing in the Harder-Narasimhan filtration of $\pi^{*}E|_{\cU_p}$ such that
\begin{align*}
    \mu^{\min}(S)+\mu^{\min}(T_{\cU/Y}|_{\cU_p})>\mu^{\max}(\pi^{*}E|_{\cU_p}/S),
\end{align*}
then there is a subsheaf $\widetilde{S}$ on $Z$ of $E$ such that $\widetilde{S}|_{\cU_p}$ is a subsheaf of $\pi^{*}E|_{\cU_p}$ agreeing with $S$ on the locus where $S$ is a vector bundle. 
\end{lemma}

\begin{proof}
By \cite[Lemmas 5 and 7]{S77}, we can replace $\cM$ by a dense open subset so that the members of the Harder-Narasimham filtration of $\pi^{*}E|_{\cU_p}$ extend to a family over $\cU$. 
Namely, there exists a sequence of subsheaves $0=S_0\subset S_1\subset \cdots S_a= \pi^{*}E$ that restricts to the Harder-Narasimham filtration of $E|_{\cU_p}$ for all $p\in \cM$, so in particular $S=S_i|_{\cU_p}$ for some $i$. 

We have an open subset $\cU^0\subset \cU$ whose complement has codimension at least 2 and consists of the points over which $S_i$ and $\pi^{*}E/S_i$ are both vector bundles. The image of $\cU^0$ in $Z$ is an open subset $Z^0$ (by flatness of $\pi$) whose complement must also have  codimension at least 2. 

Now, we can apply \Cref{lem-descent} in the case $Y=\cU^0$ and $Z=Z^0$. In order to do so, we must show that
\begin{align}\label{Descentcondition}
    \Hom(T_{\cU^0/Z^0},\Hom(S_i|_{\cU^0},(\pi^{*}E/S_i)|_{\cU^0}))=0.
\end{align}
For this, let $C \subset \cU_{p}$ be a general complete intersection curve of sufficiently high degree. Restricting the Harder-Narasimhan filtration of $\pi^{*}E|_{\cU_p}$ to $C$ will result in a sequence of vector sub-bundles. This is because each semistable subquotient on $\cU_{p}$ is in particular torsion-free, so the Harder-Narasimhan filtration is a sequence of vector subbundles away from a set of codimension at least 2. Since $C$ can be chosen to avoid this set of codimension 2, restricting a sequence of sub-bundles yields a sequence of sub-bundles. By \cite{MR81}, this sequence of sub vector bundles on $C$ is the Harder-Narasimhan filtration of $\pi^{*}E|_{C}$.  $\Hom(T_{\cU/Z}|_{\cU_p}\otimes S,\pi^{*}E|_{\cU_p}/S)=0$.   In order to show \eqref{Descentcondition}, it now suffices to show 
\begin{align*}
    \Hom(T_{\cU/Z}|_{C},\Hom(S|_C,\pi^{*}E|_{\cU_p}/S)|_C))=\Hom(T_{\cU/Z}|_{C}\otimes S|_C,\pi^{*}E|_{\cU_p}/S)|_C)=0.
\end{align*}

We conclude by applying \Cref{subquots} as 
\begin{align*}
    \mu^{\max}(\pi^{*}E|_{\cU_p}/S)|_C) &= \deg(C) \mu^{\max}(\pi^{*}E|_{\cU_p}/S))\\
    \mu^{\min}(S|_C) &= \deg(C) \mu^{\min}(C)\\
    \mu^{\min}(T_{\cU/Z}|_{C}) &= \deg(C) \mu^{\min}(T_{\cU/Z})\\
    \mu^{\min}(T_{\cU/Z}|_{C}\otimes S|_{C}) &= \mu^{\min}(S|_C) + \mu^{\min}(T_{\cU/Z}|_{C}).
\end{align*}
\end{proof}

In order for us to apply \Cref{integrate}, it is necessary to understand the sheaf $T_{\cU/Z}|_{\cU_{p}}$.  The \Cref{lem-LazMuk} and \Cref{reltmor} identifies the sheaf in two common situations.

\begin{definition}
Let $Y$ be a variety and $V$ be a globally generated vector bundle on $Y$. Then, the \emph{Lazarsfeld-Mukai} bundle of $V$ is the kernel of the evaluation map $\OO_Y\otimes H^0(V)\to V$. 
\end{definition}

\begin{lemma}
\label{lem-LazMuk}
Let $Z$ be a smooth projective variety and $\cM$ be a smooth open subset of a component of the Hilbert scheme of varieties on $Z$. Let $\cU$ be the universal family, and suppose that the natural map $\pi: \cU \to Z$ is smooth. Let $\cU_{p}$ be a general fiber of $\cU \to \cM$. Then $T_{\cU/Z}|_{\cU_{p}}$ is the Lazarsfeld-Mukai bundle for the normal bundle $N_{\cU_{p}/Z}$, defined by the short exact sequence
\[ 0 \to T_{\cU / Z}|_{\cU_{p}} \to H^0(N_{\cU_{p}/Z}) \otimes \cO_{\cU_{p}} \to N_{\cU_{p}/Z} \to 0. \]
\end{lemma}
\begin{proof}
First we compare the normal sheaf of $\cU$ in $\cM \times Z$ to the normal sheaf $N_{\cU_p/Z}$. We have the following diagram.

\catcode`\@=10
\newdimen\cdsep
\cdsep=3em

\def\cdstrut{\vrule height .25\cdsep width 0pt depth .12\cdsep}
\def\@cdstrut{{\advance\cdsep by 2em\cdstrut}}

\def\arrow#1#2{
  \ifx d#1
    \llap{$\scriptstyle#2$}\left\downarrow\cdstrut\right.\@cdstrut\fi
  \ifx u#1
    \llap{$\scriptstyle#2$}\left\uparrow\cdstrut\right.\@cdstrut\fi
  \ifx r#1
    \mathop{\hbox to \cdsep{\rightarrowfill}}\limits^{#2}\fi
  \ifx l#1
    \mathop{\hbox to \cdsep{\leftarrowfill}}\limits^{#2}\fi
}
\catcode`\@=10

\cdsep=3em
$$
\begin{matrix}
& & 0 & & 0 \cr
& & \arrow{u}{} & & \arrow{u}{} \cr
& & \cO_{\cU_p}^N & \arrow{r}{=} & \cO_{\cU_p}^N \cr
& & \arrow{u}{} & & \arrow{u}{} \cr
0 & \arrow{r}{} &  T_{\cU}|_{\cU_p}  & \arrow{r}{} & T_{\cM \times Z}|_{\cU_p} & \arrow{r}{} &  N_{\cU/ \cM \times Z}|_{\cU_p}& \arrow{r}{} & 0          \cr
& &  \arrow{u}{} & & \arrow{u}{} & & \arrow{u}{\cong} \cr
0 & \arrow{r}{} & T_{\cU_p} & \arrow{r}{} & T_{Z}|_{\cU_p} & \arrow{r}{} &  N_{\cU_p/Z} & \arrow{r}{} & 0 \cr
& & \arrow{u}{} & & \arrow{u}{} \cr
& & 0 & & 0 \cr
\end{matrix}
$$

In this diagram, we have written $H^0(N_{\cU_{p}/Z}) \otimes \cO_{\cU_{p}}$ as $\cO_{\cU_p}^N$, where $N=h^0(N_{\cU_{p}/Z})$.  We see that $N_{\cU/ \cM \times Z}|_{\cU_{p}}$ is isomorphic to $N_{\cU_{p}/Z}$ by the eight lemma. Next we relate $N_{\cU/ \cM \times Z}|_{\cU_{p}}$ to the Lazarsfeld-Mukai bundle. Consider the following diagram, where the lower right entry is computed by the eight lemma.

$$
\begin{matrix}
& & 0 & & 0 \cr
& & \arrow{u}{} & & \arrow{u}{} \cr
& & T_Z|_{\cU_p} & \arrow{r}{=} & T_Z|_{\cU_p} \cr
& & \arrow{u}{} & & \arrow{u}{} \cr
0 & \arrow{r}{} &  T_{\cU}|_{\cU_p}  & \arrow{r}{} & T_{\cM \times Z}|_{\cU_p} & \arrow{r}{} &  N_{\cU / \cM \times Z}|_{\cU_p}& \arrow{r}{} & 0          \cr
& &  \arrow{u}{} & & \arrow{u}{} & & \arrow{u}{\cong} \cr
0 & \arrow{r}{} & T_{\cU / Z} & \arrow{r}{} & \pi^{*}T_{\cM}|_{\cU_p} & \arrow{r}{} &  N_{\cU / \cM \times Z}|_{\cU_p} & \arrow{r}{} & 0 \cr
& & \arrow{u}{} & & \arrow{u}{} \cr
& & 0 & & 0 \cr
\end{matrix}
$$

Then since $\pi^* T_{\cM}$ is constant on $\cU_p$ and $N_{\cU/ \cM \times Z}|_{\cU_p} \cong N_{\cU_p/Z}$, we see that the last row becomes
\[ 0 \to T_{\cU / Z}|_{\cU_p} \to H^0(N_{\cU_p/Z}) \otimes \cO \to N_{\cU_p/Z} \to 0. \]

The result follows.

\end{proof}

\begin{lemma}
\label{reltmor}
Let $Y$ and $Z$ be smooth projective schemes and $\cM$ be an open subset of $\operatorname{Mor}(Y,Z)$. Let $\pi: Y\times\cM\to Z$ be the universal map. For $f: Y\to Z$ in $\cM$, suppose $f^{*}T_Z$ is globally generated. Then, the restriction $T_{Y\times\cM/Z}|_{Y\times\{[f]\}}$ is an extension of $T_Y$ by the Lazersfeld Mukai bundle of $f^{*}T_Z$. 
\end{lemma}

\begin{proof}
By definition, we have the short exact sequence
\begin{align*}
    0\to T_{Y\times\cM/Z}|_{Y\times\{[f]\}}\to T_{Y\times\cM/Z}|_{Y\times\{[f]\}}\to f^{*}T_Z\to 0. 
\end{align*}
We have the natural decomposition $T_{Y\times\cM/Z}|_{Y\times\{[f]\}}\cong H^0(f^{*}T_Z)\otimes \cO\oplus T_Y$, with respect to which the natural map $T_{Y\times\cM/Z}|_{Y\times\{[f]\}}\to f^{*}T_Z$ is $\operatorname{ev}+df$.   Consider the following commutative diagram, where $K$ is the Lazarsfeld-Mukai bundle of $f^{*}T_Z$:
\begin{center}
    \begin{tikzcd}
     & 0 &0 &0 &\\
     0 \ar[r]& T_Y \ar[r,equal] \ar[u] & T_Y \ar[u] \ar[r] & 0 \ar[u]\ar[r] & 0\\
     0 \ar[r]& T_{Y\times\cM/Z}|_{Y\times\{[f]\}} \ar[r]\ar[u] & H^0(f^{*}T_Z)\otimes \cO\oplus T_Y \ar[u] \ar[r,"\operatorname{ev}+df"] & f^{*}T_Z\ar[r] \ar[u] & 0\\
     0 \ar[r]& K \ar[r] \ar[u] & H^0(f^{*}T_Z)\otimes \cO \ar[u] \ar[r,"\operatorname{ev}"] & f^{*}T_Z \ar[u,equal] \ar[r]& 0\\
          & 0\ar[u] &0\ar[u] &0\ar[u] &
    \end{tikzcd}
\end{center}
The rows and columns are exact and the left column gives $T_{Y\times\cM/Z}|_{Y\times\{[f]\}}$ as an extension of $K$ by $T_Y$. 
\end{proof}

\begin{lemma}
\label{lem-O-1}
The Lazarsfeld-Mukai bundle of any globally generated vector bundle on $\PP^1$ is a direct sum of $\cO(-1)$'s.
\end{lemma}
\begin{proof}
Taking Lazarsfeld-Mukai bundles behaves well with respect to direct sum, so it remains to show the result for line bundles $\cO(a)$ with $a \geq 0$. The Lazarsfeld-Mukai bundle $M$ satisfies
\[ 0 \to M \to \cO \otimes H^0(\cO(a)) \to \cO(a) \to 0 .  \]
It follows that $M$ has rank $a$, degree $-a$ and no global sections, so that $M= \cO(-1)^a$. The result follows.
\end{proof}

\begin{proof}[Proof of Theorem \ref{thm-GrauertMulich}]
We apply \Cref{integrate} to our situation, where $\cM$ is an open subset of $\operatorname{Mor}(\mathbb{P}^1,Z)$ containing $[f]$ and $\cU=\PP^1\times \cM$. Then, applying \Cref{reltmor} shows $T_{\PP^1\times\cM/Z}|_{\PP^1\times\{[f]\}}$ is the direct sum of the Lazersfeld Mukai bundle of $f^{*}T_Z$ with $T_{\PP^1}$. By \Cref{lem-O-1}, the Lazersfeld Mukai bundle of $f^{*}T_Z$ is a sum of $\OO(-1)$'s, so $T_{\PP^1\times\cM/Z}|_{\PP^1\times\{[f]\}}$ is an extension of $\OO(2)$ by a direct sum of $\OO(-1)$'s implying $\mu^{\min}(T_{\PP^1\times\cM/Z}|_{\PP^1\times\{[f]\}}) \geq -1$. 

Suppose $f^{*}E$ splits as $\bigoplus_i \cO(a_i)$ with $a_1 \geq \cdots \geq a_r$ and $a_j \leq a_{j+1}-2$. Letting $S=\oplus _{i\leq j}\cO(a_i)$, we find 
\begin{align*}
    a_j + (-1) &>a_{j+1}\\ 
    \mu^{\min}(S)+\mu^{\min}(T_{\PP^1\times\cM/Z}|_{\PP^1\times\{[f]\}})&>\mu^{\max}((f^{*}E)/S),
\end{align*}
Therefore, we can apply  \Cref{integrate} and conclude. 
\end{proof}

\begin{proof}[Proof of \Cref{lem:slopesOfRestrictionsTokPlanes}]
This follows from \Cref{integrate} with $Z=\mathbb{P}^n$, $\cM=\mathbb{G}(k,n)$ and $\cU$ the universal $k$-plane. The only thing to check is $\mu^{\min}(T_{\cU/\mathbb{P}^n}|_\Lambda)=-\frac{1}{k}$. By \Cref{lem-LazMuk}, $T_{\cU/\mathbb{P}^n}|_\Lambda$ lies in the sequence
\begin{align*}
    0\to T_{\cU/\mathbb{P}^n}|_{\Lambda}\to H^0(\OO_{\Lambda}(1)^{n-k})\otimes \OO_{\Lambda}\to \OO_{\Lambda}(1)^{n-k} \to 0,
\end{align*}
and so is isomorphic to $\Omega_{\Lambda}(1)^{n-k}$ by the Euler sequence. Since $\Omega_{\Lambda}(1)$ is semistable with slope $-\frac{1}{k}$ \cite[Theorem 1.3.2]{OSS80}, the result follows. 
\end{proof}

\section{Plane curves}
We now apply the results from the previous section to analyze the map to moduli $\Phi$ introduced in \Cref{subsection:maptomoduli} in the case of plane curves. Throughout this section, $C$ in $\PP^2$ denotes a reduced plane curve. (In the non-reduced case, we simply pass to the reduction and apply the results of this section.) Our main results in this section are stated below.

\begin{theorem}
\label{thmplanecurves}
Let $C$ be a reduced plane curve of degree $d$. Then, the map
\begin{align*}
    \Phi: \operatorname{Mor}_e(\mathbb{P}^1,\mathbb{P}^2)&\dashrightarrow \mathbb{P}(H^0(\OO_{\PP^1}(ed)))\\
    [\iota]&\mapsto [\iota^{-1}(C)]
\end{align*}
has maximal rank if $C$ has finite stabilizer under the action of $PGL_{3}$. 
\end{theorem}
In fact, we can classify all cases in \Cref{thmplanecurves} where $\Phi$ does not have maximal rank.
\begin{theorem}
\label{classification}
Furthermore we get a complete classification of cases when $\Phi$ in \Cref{thmplanecurves} does not have maximal rank:
\begin{enumerate}
    \item $d\geq 5$: $C$ is a union of orbits under an action of $\mathbb{G}_m$ or $\mathbb{G}_a$ on $\PP^2$
    \item $d=4$: 
    \begin{enumerate}
        \item $e=1$ and $C$ is the union of four concurrent lines
        \item $e\geq 2$ and $C$ is a union of orbits under an action of $\mathbb{G}_m$ or $\mathbb{G}_a$ on $\PP^2$
    \end{enumerate}
    \item $d=3$: $e\geq 2$ and $C$ is union of concurrent lines.
\end{enumerate}
\end{theorem}

Before giving the proofs of these theorems, we need the following two propositions.

\begin{proposition}
\label{O1vectorfield}
If $C$ is a reduced plane curve and $T_{\PP^2}(-\log C)$ admits a nontrivial homomorphism from $\OO_{\PP^2}(1)$, then $C$ is a union of concurrent lines.
\end{proposition}

\begin{proof}
First, a nontrivial map from $\OO_{\PP^2}(1)\to T_{\PP^2}(-\log C)$ induces a nontrivial map $\OO_{\PP^2}(1)\to T_{\PP^2}$. Consider the Euler sequence
\begin{align*}
    0\to \OO_{\PP^2}\to \OO_{\PP^2}(1)^3\to T_{\PP^2} \to 0.
\end{align*}
Applying $\Hom(\OO_{\PP^2}(1),\cdot)$ to the Euler sequence, we find $\Hom(\OO_{\PP^2}(1),T_{\PP^2})\cong \Hom(\OO_{\PP^2}(1),\OO_{\PP^2}(1)^3)$ and that the composite map $\OO_{\PP^2}(1)\to T_{\PP^2}(-\log C) \to T_{\PP^2}$ lifts to a map $\OO_{\PP^2}(1)\to \OO_{\PP^2}(1)^3$.

After a change of coordinates, we can assume the map $\OO_{\PP^2}(1)\to \OO_{\PP^2}(1)^3$ is inclusion into the first factor. 
The map $\OO_{\PP^2}(1)^3\to T_{\PP^2}$ sends a tuple of linear forms $(L_1,L_2,L_3)$ to $(L_1\frac{\partial}{\partial_X},L_2\frac{\partial}{\partial_Y},L_3\frac{\partial}{\partial_Z})$, so we conclude that $L\frac{\partial}{\partial_X}$ is a section of $T_{\PP^2}(-\log C)$ for all linear forms $L$. 

In particular, away from the point $[1:0:0]$, the tangent vector $\frac{\partial}{\partial_X}$ is in the tangent space of $C$ for every point of $C$. Restricting to the affine chart $\{Z\neq 0\}$ with coordinates $(x,y)$ and dehomogenizing, this means $C$ restricts to a union of lines parallel to the $x$-axis. Since these lines and the line at infinity are precisely the lines passing through $[1:0:0]$, we conclude $C$ is a union of concurrent lines. 
\end{proof}

\begin{proposition}
\label{integralcurve}
If $C$ is a reduced plane curve and $T_{\PP^2}(-\log C)$ has a section, then $C$ is equivalent to a union of orbits under one of the  two actions by $\mathbb{G}_m$ and $\mathbb{G}_a$ as follows:
\begin{align*}
    \mathbb{G}_m &\to GL_3\\
    t &\mapsto \begin{pmatrix} t^a & 0 & 0\\ 0 & t^b & 0 \\ 0 & 0 & 1\end{pmatrix} \qquad a,b\in\mathbb{N} \\
       \mathbb{G}_a&\to GL_3\\
    t&\mapsto \operatorname{exp}\left(t\begin{pmatrix} 0 & 1 & 0\\ 0 & 0 & 1\\ 0 & 0 & 0\end{pmatrix}\right) =\begin{pmatrix} 1 & t & \frac{1}{2}t^2\\ 0 & 1 & t \\ 0 & 0 & 1\end{pmatrix}.
\end{align*}
Explicitly, there are two cases:
\begin{enumerate}
    \item $C$ is projectively equivalent to a union of curves of the form $X^{p} Y^{q}=cZ^{p+q}$, $c\in \mathbb{C}^{\times}$, and possibly a subset of the three coordinate lines.
    \item $C$ is projectively equivalent to a union of members of the family $\{XZ-Y^2+cZ^2\mid c\in \mathbb{C}\}$ of conics quadritangent to $\{XZ-Y^2=0\}$ at $[0:0:1]$, and possibly the line $\{Z=0\}$. 
\end{enumerate}
\end{proposition}

\begin{proof}
Let $s$ be a section of $T_{\PP^2}(-\log C)$.  Then, $s$ is also a section of $T_{\PP^2}$ and can be written as $L_X\frac{\partial}{\partial_X}+L_Y\frac{\partial}{\partial_Y}+L_Z\frac{\partial}{\partial_Z}$ where $L_X, L_Y, L_Z$ are homogenous linear forms in $X$, $Y$ and $Z$. 

Let $C_0$ be a component of $C$ and let $p\in C$ be a smooth point of $C_0$.  We lift $p\in \mathbb{P}^2$ to a point $\widetilde{p}\in \mathbb{C}^3 \setminus \{0\}$. 
Then, $C_0$ contains the projection under $\mathbb{C}^3\setminus\{0\}\to \mathbb{P}^2$ of the integral curve $\widetilde{C}$ through $\widetilde{p}$ which is the solution to the matrix differential equation
\begin{align}
\label{21b}
    \frac{d}{dt} \begin{pmatrix} X(t) \\ Y(t) \\ Z(t)\end{pmatrix} &= A \begin{pmatrix} X(t) \\ Y(t) \\ Z(t)\end{pmatrix},\quad \begin{pmatrix} X(0) \\ Y(0) \\ Z(0)\end{pmatrix} =\widetilde{p}
\end{align}
 where $A$ is the 3 by 3 matrix with complex entries such that 
\begin{align*}
    A\begin{pmatrix} X \\ Y \\ Z\end{pmatrix} &=\begin{pmatrix} L_X(X,Y,Z) \\L_Y(X,Y,Z)\\ L_Z(X,Y,Z)\end{pmatrix}.
\end{align*}
 If the projection of $\widetilde{C}$ to $\mathbb{P}^2$ is not a single point, then the image is dense in $C_0$. Therefore, $C$ must be (the closure of) a finite union of projections of integral curves in $\mathbb{C}^3 \setminus \{0\}$ and 1-dimensional components of the zero locus of $s$. 

After a linear change of coordinates, we can assume that $A$ is in Jordan block form. We keep this choice of coordinates from now on. We let $\widetilde{p}=(c_1,c_2,c_3) \in \mathbb{C}^{3}$ denote a lift of a point on $C_{0}$ (to be determined separately in each case)  and we let $P(X,Y,Z)$ be a homogenous polynomial defining $C_0$.

\textbf{Case 1: $A$ is diagonal.} In this case we will show that the first case of \eqref{21b} happens, so we can assume that $c_1,c_2,c_3\neq 0$ or else $\widetilde{p}$ is contained in a coordinate line. Let $\lambda_1,\lambda_2,\lambda_3$ be the eigenvalues of $A$. Then, the solution to \eqref{21b} is $\begin{pmatrix} X(t) \\ Y(t) \\ Z(t)\end{pmatrix}=\begin{pmatrix} e^{\lambda_1 t} c_1 \\ e^{\lambda_2 t} c_2 \\ e^{\lambda_3 t} c_3\end{pmatrix}$.

The defining equation $P(X,Y,Z)$ of $C_{0}$ is a homogenous polynomial of minimial degree satisfying
\begin{align}
    P(c_1e^{\lambda_1 t},c_2e^{\lambda_2 t},c_3e^{\lambda_3 t})=0.
\end{align} 
We can choose a new grading on  $\mathbb{C}[X,Y,Z]$ by the complex numbers $\mathbb{C}$ where the monomial $X^a Y^b Z^c$ has the grade $a\lambda_1+b\lambda_2+c\lambda_3$. Let $P_\omega$ be the homogenous component of $P$ with grade $\omega\in \mathbb{C}$. By linear independence of characters, the elements in $\{e^{\omega t} \mid \omega \in \mathbb{C}\}$ are linearly independent, and hence $P_{\omega}(c_1e^{\lambda_1 t},c_2e^{\lambda_2 t},c_3e^{\lambda_3 t})=0$. Therefore, $P$ divides $P_{\omega}$ for all $\omega\in \mathbb{C}$ so $P_{\omega}$ can be nonzero for only one value of $\omega$. 

We cannot have $\lambda_1,\lambda_2,\lambda_3$ all equal or else $s$ would be a multiple of $X\frac{\partial}{\partial_X}+Y\frac{\partial}{\partial_Y}+Z\frac{\partial}{\partial_Z}$ which induces the zero vector field on $\PP^2$. The monomials $X^{a}Y^{b}Z^{c}$ that can appear in $P$ with nonzero coefficients must be the solution to the two linear equations
\begin{align}
\label{twoeq}
    a+b+c&=\deg(C_0)\\
    \lambda_1a + \lambda_2 b+\lambda_3 c &= \omega
\end{align}
for some fixed $\omega$. The solution set to \eqref{twoeq} is some 1-dimensional complex line $\ell$ in $\mathbb{C}^3$ and we are are interested in the integer solutions $\ell\cap \mathbb{Z}^3$. If $\ell\cap \mathbb{Z}^3$ is empty or a single point, then $P$ is a monomial, hence degree 1 by irreducibility. So the only remaining case is if $\ell\cap \mathbb{Z}^3$ is a 1-dimensional lattice, which can be written in the form $\{(a_0,b_0,c_0)+m(a_1,b_1,c_1)\mid m\in \mathbb{Z}\}$. 

Thus, we know that the monomials $X^{a}Y^b Z^c$ that can appear with nonnegative coefficients in $P$ must be in $S=\{(a_0,b_0,c_0)+m(a_1,b_1,c_1)\mid m\in \mathbb{Z}\}\cap \mathbb{Z}_{\geq 0}^3$. If $S$ contains exactly one element, then $P$ is degree one by irreducibility. If $S$ contains exactly two elements, then $P$ is a binomial and must then be of the form $X^aY^b+kZ^{a+b}$ for some $k\neq 0$, because $P$ is irreducible. Finally, one can check $S$ cannot contain three or more elements assuming $P$ is irreducible. 

\textbf{Case 2: $A$ has exactly two Jordan blocks} Let $\lambda_1$ be the eigenvalue of the $2\times 2$ block and $\lambda_2$ be the eigenvalue of the $1\times 1$ block. In this case we will show that the first case of \Cref{integralcurve} happens. We can assume $C_0$ is not contained in a coordinate line, and therefore assume $\widetilde{p}$ is such that $c_2,c_3\neq 0$.  Then, a solution to \eqref{21b} is $\begin{pmatrix} X(t) \\ Y(t) \\ Z(t)\end{pmatrix}=\begin{pmatrix} e^{\lambda_1 t} c_1 + c_2 te^{\lambda_1 t} \\ e^{\lambda_1 t} c_2 \\ e^{\lambda_2 t} c_3\end{pmatrix}$ for $\widetilde{p}=(c_1,c_2,c_3)$.

This means $P(e^{\lambda_1 t} c_1 + c_2 te^{\lambda_1 t} , e^{\lambda_1 t} c_2 , e^{\lambda_2 t} c_3)=0$. Dividing by $e^{\deg(P)\lambda_1t}$ and letting $\lambda = \lambda_2-\lambda_1$, we find
\begin{align*}
    P(c_1+c_2 t, c_2, e^{\lambda t} c_3)=0.
\end{align*}
Reparameterizing $t$ by $t-\frac{c_1}{c_2}$, we can assume $c_1=0$. We claim now that the map $\mathbb{C}[X,Y,Z]\to \mathbb{C}[[t]]$ sending $P(X,Y,Z)$ to $P(c_2 t, c_2, e^{\lambda t} c_3)$ is an injection because $c_2 t, c_2, e^{\lambda t} c_3$ are algebraically independent. The latter claim follows from the fact that the  functions $\{t^me^{\omega t}\mid m\in \mathbb{Z}_{\geq 0}, \omega \in \mathbb{C}\}$ are linearly independent. Therefore, $P=0$, i.e. $C_0$ must be contained in either the $\{Y=0\}$ or $\{Z=0\}$ coordinate lines, establishing this case.  

\textbf{Case 3: $A$ has exactly one Jordan block}
Let $\lambda$ be the unique eigenvalue of $A$. Subtracting a diagonal matrix from $A$ is equivalent to subtracting the Euler vector field $X\frac{\partial}{\partial_X}+Y\frac{\partial}{\partial_Y}+Z\frac{\partial}{\partial_Z}$ from the vector field $s$, so we can assume $\lambda=0$. Then, a solution to \eqref{21b} is $\begin{pmatrix} X(t) \\ Y(t) \\ Z(t)\end{pmatrix}=\begin{pmatrix}  c_1 + c_2 t + c_3 \frac{1}{2}t^2 \\  c_2+c_3t \\c_3\end{pmatrix}$ for $\widetilde{p}=(c_1,c_2,c_3)$. In this case we will show that the second case of \Cref{integralcurve} happens, so we can assume that $c_3\neq 0$ or else $C_{0}$ is contained in $\{Z=0\}$.

We know that $P( c_1 + c_2 t + c_3 \frac{1}{2}t^2 , c_2+c_3t , c_3)=0$. Now we change coordinates on $t$. Letting $t\mapsto t-\frac{c_2}{c_3}$ yields $P( c_1 +\frac{1}{2}\frac{c_2^2}{c_3} + c_3 \frac{1}{2}t^2, c_3t, c_3)=0$. Dividing out by a power of $c_3$ and replacing $c_1$ with another constant $c_1'$, we find $P( c_1'+ \frac{1}{2}t^2, t, 1)=0$. 

As $t$ varies, the curve $(c_1'+ \frac{1}{2}t^2, t, 1)$ parameterizes the conic $XZ-\frac{1}{2}Y^2-c_1'Z^2$ in $\PP^2$, settling this case. 
\end{proof}

\begin{proof}[Proofs of \Cref{thmplanecurves} and \Cref{classification}]
We will prove \Cref{classification} which implies \Cref{thmplanecurves}. Let $f: \mathbb{P}^1\to \mathbb{P}^2$ be a general map of degree $e$. The log tangent sheaf $T_{\PP^2}(-\log C)$ is a vector bundle since it is a reflexive sheaf on a surface. Pulling back $T_{\PP^2}(-\log C)$ to $\PP^1$ yields a rank 2 vector bundle $E$ of degree $(3-d)e$. We split our analysis into cases.

\textbf{Case: $d\geq 5$ or $d=4$ and $e\geq 2$.} If $\Phi$ is not of maximal rank, then $E\cong \OO(a)\oplus \OO(b)$ where $a\geq 0$. Since the total degree of $E$ is at most $-2$, we get $a-b\geq 2$ and we can apply \Cref{thm-GrauertMulich} to find a line subbundle of $T_{\PP^2}(-\log C)$ of non-negative degree. This means $T_{\PP^2}(-\log C)$ has a section and we conclude by \Cref{integralcurve}.

\textbf{Case: $d=4$ and $e=1$.} If $\Phi$ is not of maximal rank, then $h^0(E)\geq 2$ in this case. This means $E\cong \OO(a)\oplus \OO(b)$ where $a\geq 1$. In this case $a-b\geq 3$, so we can apply \Cref{thm-GrauertMulich} to find a line subbundle $\OO(a)$ of  $T_{\PP^2}(-\log C)$. Applying \Cref{O1vectorfield}, we are done in this case.

\textbf{Case: $d=3$ and $e\geq 2$.} In this case, $\deg(E)=0$ and a dimension count shows that $\Phi$ is not of maximal rank whenever $h^0(E)\geq 3$. Hence, we can apply \Cref{thm-GrauertMulich} to find a line subbundle of $T_{\PP^2}(-\log C)$ of positive degree, so we again conclude using \Cref{O1vectorfield}. 

\textbf{Case: $d=3$ and $e=1$.} We can find a line $\ell$ meeting $C$ in three distinct points. This means $\Phi$ is automatically surjective, so it is of maximal rank. 

\textbf{Case: $d=2$.} In this case, $\deg(E)=e$ and $\Phi$ is not of maximal rank iff $E\cong O(a)\oplus O(b)$ where $a\geq e+2$ and $b\leq -2$. Applying \Cref{thm-GrauertMulich}, we find a line subbundle of $T_{\PP^2}(-\log C)$ of degree at least $\lceil\frac{e+2}{e}\rceil=2$. However, there are no nontrivial maps $\OO(2)\to T_{\PP^2}$, showing $\Phi$ must have maximal rank. 
\end{proof}

\section{Hyperplane sections}

We let $X$ be a smooth degree $d$ hypersurface in $\PP^n$. 
Using the notation from \cref{subsection:maptomoduli}, our objective is to prove that $\Phi$ has maximal rank when $k = n-1$ and $e=1$.  Unlike the plane curve case, we are unable to obtain a complete  classification statement like \Cref{classification}. However, we are able to prove that if $d$ is larger than $n+1$, the hyperplane sections of $X$ vary maximally in moduli. In this section, we prove \Cref{thm:hypsliceintro2} and some generalizations, captured below in \Cref{thm:knboundslice} and \Cref{thm:smallkslice}.

Our results all rely on a stability result from Guenancia \cite{guenancia}. The following version comes from Guenancia's Theorem A by observing that the canonical bundle of a degree $d$ hypersurface in $\PP^n$ is ample when $d \geq n+2$.

\begin{theorem}[{Thm A from \cite{guenancia}}]
\label{theorem:Guenancia}
If $X$ is a smooth hypersurface of degree $d \geq n+2$, then $T_{\PP^n}(-\log X)$ is semistable.
\end{theorem}


Using Theorem \ref{theorem:Guenancia}, the basic strategy is to understand how large the degree $d$ can be such that the restricton of $T_{\PP^n}(-\log X)$ to the curve or $k$-plane can have a section. We use results from Section \ref{sec:Grauert-Mulich} to do this.

\begin{theorem}
\label{theorem:degEcurvesLargen}
If $X$ in $\mathbb{P}^n$ is a smooth hypersurface of degree $d$, then the space of degree $e$ rational curve sections of $X$ vary maximally in modulus when $d > \frac{n(n-1)}{2e} + n + 1$.
\end{theorem}
\begin{proof}
Consider the bundle $T_{\PP^n}(-\log X)$. By Theorem \ref{theorem:Guenancia}, this bundle is semi-stable. For $d$ larger than $n+1$, we see that a section of this bundle would give a destabilizing subsheaf, so we know that $T_{\PP^n}(-\log X)$ has no sections.

Let $\cM=\operatorname{Mor}_e(\PP^1,\PP^n)$ be the space of parameterized degree $e$ rational curves in $\PP^n$. Given a choice of $F$ with $X=V(F)$, there is a natural map $\Phi: \cM \to H^0(\OO_{\PP^1}(ed))$ sending a map $f:\PP^1\to \PP^n$ to the pullback $f^{*}F\in H^0(\PP^1,\OO_{\PP^1}(de))$. We know by Lemma \ref{logtangent} that the tangent space to the fiber of $\Phi$ at a given map $f:\PP^1\to \PP^n$ is simply $H^0(f^{*}T_{\PP^n}(-\log X))$. To show that $\phi$ is generically finite, we need only show that $h^0(f^{*}T_{\PP^n}(-\log X)) = 0$.

By Theorem \ref{thm-GrauertMulich}, we see that $f^{*}T_{\PP^n}(-\log X)$ is a direct sum of line bundles $\bigoplus \OO(a_i)$ with consecutive $a_i$ differing by at most 1. Thus, any such bundle on $\PP^1$ that has a section will have degree larger than that of the bundle $\OO \oplus \OO(-1) \oplus \cdots \oplus \OO(-n+1)$. From this it follows that any semi-stable bundle $E$ on $\PP^n$ such that $f^{*}E$ has a section for a general map $f:\PP^1\to \PP^n$ will have degree at least $-\frac{n(n-1)}{2}$. Thus, if we want $h^0(f^{*}T_{\PP^n}(-\log X)) = 0$, we must have that $\deg f^{*} T_{\PP^n}(-\log X) \geq - \frac{n(n-1)}{2}$. Since $\deg f^{*}T_{\PP^n}(-\log X) = e(n+1-d)$, the result follows.
\end{proof}

We now consider $k$-plane sections of smooth hypersurfaces. By Lemma \ref{lem:slopesOfRestrictionsTokPlanes}, we need to understand torsion free sheaves on $\PP^k$ whose Harder-Narasimhan filtration has sub-quotients whose slopes do not decrease too quickly, namely $\mu_1>\mu_2>\cdots\mu_a$ with $\mu_{i}-\mu_{i+1}\leq \frac{1}{k}$ for all $i$. Understanding the possible slopes that may appear in the Harder-Narasimhan filtration is a combinatorially interesting problem, which we describe below.

\begin{definition}
Let a sequence $(d_1,r_1),(d_2,r_2),\ldots, (d_a,r_a)$ in $\mathbb{Z}_{\geq 0}\times \mathbb{Z}_{>0}$ be \emph{$k$-admissible} if $d_1\leq 0$ and $0\leq \frac{d_{i+1}}{r_{i+1}}-\frac{d_i}{r_i}\leq \frac{1}{k}$ for each $i$. Let $A_{k,n}$ denote the set of $k$-admissible sequences with $\sum_{i} r_i=n$ (where $a$ is arbitrary).
\end{definition}
\begin{definition}
 Define $B_k(n)$ to be $\max\{\sum_{i=1}^{a}{d_i}\mid (d_1,r_1),\ldots(d_a,r_a)\text{ in }A_{k,n}\}$.
\end{definition}

\begin{lemma}
\label{lem:PhikrelationtoDegree}
If $E$ is a semistable sheaf on $\PP^n$ of rank $n$ such that its restriction to a general $k$-plane has a section, then $\deg E \geq -B_k(n)$.
\end{lemma}
\begin{proof}
Let $\Lambda$ be a general $k$-plane and $0 = E_0 \subset E_1 \subset \cdots \subset E_a = E|_{\Lambda}$ be the Harder-Narasimhan filtration of $E|_{\Lambda}$. Let $-d_i$ be the degree of $E_i/E_{i-1}$ and $r_i$ be the rank of $E_i/E_{i-1}$. Since $E|_{\Lambda}$ has a section, we see that $d_1 \leq 0$. Since $E$ is semistable, by Lemma \ref{lem:slopesOfRestrictionsTokPlanes} it follows that the sequence $(-d_1,r_1), \cdots, (-d_a,r_a)$ will be $k$-admissible. The result follows.
\end{proof}
We can compute a bound for when $k$-plane sections of a degree $d$ hypersurface in $\PP^n$ will vary maximally in moduli in terms of $B_k(n)$. 

\begin{theorem}
\label{thm:Bbound}
Let $X$ be a smooth, degree $d$ hypersurface in $\PP^n$ with $d>B_k(n)+n+1$. Then,
\begin{align*}
    \Phi: \operatorname{Mor}_1(\PP^k,\PP^n)&\dashrightarrow \PP H^0(\PP^k,\OO_{\PP^k}(d))\\
    \iota &\mapsto [\iota^{-1}(X)].
\end{align*}
is of maximal rank.
\end{theorem}
\begin{proof}
 By Theorem \ref{theorem:Guenancia}, $T_{\PP^n}(-\log X)$ will be semi-stable. By \Cref{lem:slopesOfRestrictionsTokPlanes}, $T_{\PP^n}(-\log X)|_{\Lambda}$ will have Harder-Narasimhan filtration as described in the statement of the theorem. 
Given a hypersurface $X$ together with a choice of defining equation $f$, we get a map $\phi: \operatorname{Mor}_1(\PP^k,\PP^n) \to H^0(\OO_{\PP^k}(d))$ sending a $k$-plane to the pull-back of $f$ by the $k$-plane. We wish to show that $\phi$ is generically finite. 
 
 To get a contradiction, suppose $\phi$ has only positive-dimensional fibers. By Lemma \ref{logtangent}, the tangent space to a fiber of $\phi$ at a general point $\Lambda$ is $H^0(T_{\PP^n}(-\log X)|_{\Lambda})$, so we know that $T_{\PP^n}(-\log X)|_{\Lambda}$ has a global section. Thus, by Theorem \ref{theorem:Guenancia} and Lemma \ref{lem:PhikrelationtoDegree}, the degree of $T_{\PP^n}(-\log X)$ will be at least $- B_k(n)$. It follows that 
 $$ n+1-d \geq - B_k(n). $$
This is impossible given the assumptions in the statement of the theorem.
\end{proof}

Then, \Cref{thm:hypsliceintro2} follows from the following result on $B_k(n)$.

\begin{proposition}
\label{prop:k2n/3}
If $k \geq \frac{2n}{3}$, then $B_k(n) = 1$.
\end{proposition}
\begin{proof}
Let $(d_1, r_1), \cdots, (d_a,r_a)$ be an admissible sequence of total degree $B_k(n)$. Without loss of generality, we may assume $d_1=0$, $d_i > 0$ for $i > 1$. Then it follows that $r_2 \geq k$, since $\frac{d_2}{r_2} \leq \frac{1}{k}$. Since $\frac{d_3}{r_3} \leq \frac{2}{k}$, we see that $r_3 \geq \frac{k}{2}$, provided that there are at least three terms in the sequence. However, in this case, $r_1+r_2+r_3 \geq 1+k+\frac{k}{2} = 1+\frac{3k}{2} > n$, which is impossible. Thus, $a \leq 2$.

Next, we observe that $d_2 \leq 1$, since if $d_2 \geq 2$, then $r_2 \geq d_2 k \geq 2k > n$, a contradiction. It follows that the sum of the $d_i$ is at most $1$, and since we know that 1 is achievable with the admissible sequence $(0,n-k), (1, k)$, the result follows.
\end{proof}

We defer more detailed analysis of $B_k(n)$ to \Cref{Phiknsection}. From the results in \Cref{Phiknsection} and \Cref{thm:Bbound} we get the following results.

\begin{theorem}
\label{thm:knboundslice}
If $X\subset\PP^n$ is a smooth hypersurface of degree $d$ with $d>4(\frac{n^2}{k^{\frac{3}{2}}}+k^{\frac{3}{2}})$, then the map \begin{align*}
    \Phi: \operatorname{Mor}_1(\PP^k,\PP^n)&\dashrightarrow \PP H^0(\PP^k,\OO_{\PP^k}(d))\\
    \iota &\mapsto [\iota^{-1}(X)].
\end{align*}
is of maximal rank.
\end{theorem}
\begin{proof}
This follows from \Cref{Phibound}, where it is shown $B_k(n) \leq 3 \left(\frac{n^2}{k^{\frac{3}{2}}}+k^{\frac{3}{2}} \right)$. To finish, one checks that $\frac{n^2}{k^{\frac{3}{2}}}+k^{\frac{3}{2}}\geq 2n\geq n+2$. This follows from the AM-GM equality and the fact $n\geq 2$. 
\end{proof}

In \Cref{thm:knboundslice}, we prioritized giving a clean statement and proof over giving an optimal constant. Still, one can wonder what the optimal constant by computing $B_k(n)$ for small $k$ and all $n$. In this case, \Cref{cor:manyquadrics} gives the following result:
\begin{theorem}
\label{thm:smallkslice}
For $k\leq 5$. Then, there exists a linear function $\ell(n)$ and an integer $C_k$ such that $|B_k(n)-\frac{ n^2}{C_k}|\leq \ell(n)$. Here, $C_2=3,C_3=7,C_4=11,C_5=19$. In particular, the map
\begin{align*}
    \Phi: \operatorname{Mor}_1(\PP^k,\PP^n)&\dashrightarrow \PP H^0(\PP^k,\OO_{\PP^k}(d))\\
    \iota &\mapsto [\iota^{-1}(X)].
\end{align*}
is of maximal rank if $X\subset \PP^n$ is smooth and has degree $d\geq C_k n^2+\ell(n)+n+2$.
\end{theorem}

We expect \Cref{thm:smallkslice} to hold for all values of $k$, but we can only check a finite number of cases with a computer. Roughly up to $k=100$ is what is reasonable with our methods. 

Given \Cref{thm:smallkslice}, one can ask how fast $C_k$ grows with $k$. We trivially know $C_k=O(k^2)$ by relaxing the condition that the $d_i$ are integers in the definition of an admissible sequence to compute $B_k(n)$ (in which case we let all the $r_i$ be equal to 1). We also get $C_k=\Omega(k^{\frac{3}{2}})$ from \Cref{Phibound}. From experimental evidence, we think that the actual answer is strictly between $k^{\frac{3}{2}}$ and $k^{2}$ but closer to $k^{\frac{3}{2}}$. 
\appendix
\section{Bounds and computations for $B_k(n)$}
\label{Phiknsection}
We will bound $B_k(n)$ for all $k,n$ in \Cref{Phibound}. We also compute $B_{k}(n)$ for $k$ small and arbitrary $n$, and give some conjectures about $B_{k}(n)$ in general. 

To give an idea of how the function $B_k(n)$ behaves, we note the results in \Cref{Phiknsection} can show $B_5(39)=39$, corresponding to the admissible sequence $$(0,1)(1,5)(1,3)(1,2)(2,3)(4,5)(1,1)(6,5)(4,3)(3,2)(5,3)(9,5)(2,1).$$

There are a couple of features of this admissible sequence we believe hold in general that we will only prove in special cases. First, this admissible sequence can be generated greedily, where we use greed to maximize the ratio $\frac{d}{r}$ of the last piece of the sequence. Second, the admissible sequence is essentially periodic in that the $(1,1)(6,5)(4,3)(3,2)(5,3)(9,5)$ is obtained from $(0,1)(1,5)(1,3)(1,2)(2,3)(4,5)$ by replacing each $(d,r)$ with $(d+r,r)$. We give a finite criterion that can be applied to show both the greedy property and the periodicity in \Cref{periodic} in \Cref{Phiknsection} below. 

We expect there are many other interesting patterns that can be found. For example, the segment $(0,1)(1,5)(1,3)(1,2)(2,3)(4,5)(1,1)$ of the admissible sequence above is preserved under reversing the order and replacing each $(d,r)$ with $(r-d,r)$. This pattern continues to hold for larger $k$ and suggests that these optimal admissible sequences can also be generated greedily backwards as well as forwards. 

\begin{proposition}
\label{Phibound}
We have $B_k(n) \leq 3 \left(\frac{n^2}{k^{\frac{3}{2}}}+k^{\frac{3}{2}} \right)$
\end{proposition}
\begin{proof}
Let $(d_1,r_1), \cdots, (d_a,r_a)$ be an admissible sequence with $\sum_i d_i = B_k(n)$. Let $\mu_i = \frac{d_i}{r_i}$. Let $n(j)$ be the sum of the $r_i$ such that $\mu_i\in [j-1,j)$.

Since the $\mu_i$ contributing to $n(j)$ are all less than $j$, we observe that
\[ B_k(n) \leq \sum_{j=1}^\infty jn(j)  .\]
Thus, understanding the $n(j)$ allows us to bound $B_k(n)$. The sum of all of the $n(j)$ is $n$. Let $J$ be the last nonzero $n(j)$, so $B_k(n) \leq \sum_{j=1}^{J} j n(j)$.

Let $n_k^{\min}$ be a positive number that is at most $n(j)$ for any $j < J$. Then we obtain an upper bound for $B_k(n)$
\begin{align*}
     B_k(n)& \leq \sum_{i=1}^J jn(j)\leq n_k^{\min}+2n_k^{\min}+\cdots+\lceil \frac{n}{n_k^{\min}}\rceil n_k^{\min}\\
     &= n_k^{\min} \frac{\lceil \frac{n}{n_k^{\min}}\rceil(1+\lceil\frac{n}{n_k^{\min}}\rceil)}{2}\\
     &\leq n_k^{\min} \frac{(\frac{n}{n_k^{\min}}+1)(\frac{n}{n_k^{\min}}+2)}{2}
\end{align*}

Thus, it remains to give a bound for $n_k^{\min}$. Fix $j<J$ and let $(d_{i_j+1},r_{i_j+1}),\ldots,(d_{i_j+c(j)},r_{i_j+c(j)})$ be the part of the admissible sequence with slopes $\frac{d_{i_j+1}}{r_{i_j+1}},\ldots,\frac{d_{i_j+c(j)}}{r_{i_j+c(j)}}$ in $[j-1,j)$. By definition, $r_{i_j+1}+\cdots+r_{i_j+c(j)}=n(j)$. First, we show $c(j)\geq k$. To see this, we first note that $\frac{d_{i_j+1}}{r_{i_j+1}}< (j-1)+\frac{1}{k}$. If $j=1$, then this is true because $\frac{d_{i_j+1}}{r_{i_j+1}}\leq 0$ by definition. If $j>1$, then this is true because $\frac{d_{i_j+1}}{r_{i_j+1}}\leq \frac{d_{i_j}}{r_{i_j}}+\frac{1}{k}<(j-1)+\frac{1}{k}$. 

Since $\frac{d_{i_j+1}}{r_{i_j+1}}< (j-1)+\frac{1}{k}$, we then see that 
\begin{align*}
    \frac{d_{i_j+2}}{r_{i_j+2}}&\leq \frac{d_{i_j+1}}{r_{i_j+1}}+\frac{1}{k}< (j-1)+\frac{2}{k}\\
    \vdots\quad &\quad  \vdots\\
    \frac{d_{i_j+k}}{r_{i_j+k}}&\leq  \frac{d_{i_j+k-1}}{r_{i_j+k-1}}+\frac{1}{k}<j,
\end{align*}
so $c(j)\geq k$. 

Now, we want to bound $r_{i_j+1}+\cdots+r_{i_j+c(j)}=n(j)$. In the multi-set $\{r_{i_j+1},\ldots,r_{i_j+c(j)}\}$, we know that there is at most element that is equal to 1, fewer than two elements that are equal to 2, fewer than three elements that are equal to 3 and so on. Therefore, if $m$ is the largest integer such that $1+(1+\cdots + m-1) = 1+\frac{m(m-1)}{2}$ is at most $k$, then $\frac{(m-1)^2}{2}<\frac{m(m_1)}{2}+1\leq k$, so $m \leq \sqrt{2k}+1$. Thus,
\begin{align*}
    n(j) = r_{i_j+1}+\cdots+r_{i_j+c(j)}&\geq 1+(2\cdot 1+3\cdot 2+\cdots m\cdot(m-1))\\
    &= 1+2(\binom{2}{2}+\binom{3}{2}+\cdots+\binom{m}{2})\\
    &= 1+ \frac{(m+1)m(m-1)}{3}\\
    &\geq \frac{(\sqrt{2k}+2)(\sqrt{2k}+1)\sqrt{2k}}{3}>\frac{2\sqrt{2}}{3}k^{\frac{3}{2}}.
\end{align*}
Thus, choosing $n_k^{\min}$ to be $\frac{2\sqrt{2}}{3}k^{\frac{3}{2}}$ suffices. Plugging into our earlier bound, we get an upper bound for $B_k(n)$ as 
\begin{align*}
    n_k^{\min} \frac{(\frac{n}{n_k^{\min}}+1)(\frac{n}{n_k^{\min}}+2)}{2}=\frac{n^2}{2n_k^{\min}}+ \frac{3n}{2}+n_k^{\min}=\frac{n^2}{k^{\frac{3}{2}}}\frac{9}{4\sqrt{2}}+\frac{3n}{2}+\frac{2\sqrt{2}}{3}k^{\frac{3}{2}},
\end{align*}
which is at most $2(\frac{n^2}{k^{\frac{3}{2}}}+n+k^{\frac{3}{2}})$. Applying the AM-GM inequality yields
\begin{align*}
    \frac{n^2}{k^{\frac{3}{2}}}+k^{\frac{3}{2}}\geq 2n,
\end{align*}
yielding the claimed bound.

\end{proof}

We now move on to computing exact values of $B_k(n)$ for small $k$. Our strategy is a recursive algorithm that requires some conditions to be met, and we suspect that these conditions are always met. In the course of our proof, we will use the three quantities $\mu^{\max}(n)$, $B_k^{\up}(n)$ and $B_k^{\low}(n)$. We define $\mu^{\max}(n)$ by
\[ \mu^{\max}(n) := \max\{\frac{r_a}{d_a}\mid (d_1,r_1),\ldots(d_a,r_a)\text{ in }A_n\}. \]

\begin{lemma}
\label{lem:muMaxComputation}
We can compute $\mu^{\max}(n)$ inductively by $\mu^{\max}(1) = 0$ and
\[ \mu^{\max}(n) = \max\{\frac{\lfloor(\mu^{\max}(i)+\frac{1}{k})(n-i)\rfloor}{n-i}\mid 0<i<n\} .\]
\end{lemma}
\begin{proof}
Let $\mu^{\max'}_n = \max\{\frac{\lfloor(\mu^{\max}(i)+\frac{1}{k})(n-i)\rfloor}{n-i}\mid 0<i<n\}$.

We use induction. The base case $n=1$ is vacuous, so assume $n>1$. First, we show that $\mu^{\max\prime}_n\leq \mu^{\max}_n$. Given any $0<i<n$, $\frac{\lfloor(\mu^{\max\prime}_i+\frac{1}{k})(n-i)\rfloor}{n-i}$ is a slope achieved by taking an admissible sequence $(d_1,r_1),\ldots,(d_a,r_a)$ in $A_i$ and appending $(\lfloor(\mu^{\max\prime}_i+\frac{1}{k})(n-i)\rfloor,n-i)$. So by definition $\mu^{\max\prime}_n\leq \mu^{\max}_n$. 

Now we show $\mu^{\max\prime}_n\geq \mu^{\max}_n$. Let $(d_1,r_1),\ldots(d_a,r_a)$ be an admissible sequence in $A_n$ achieving $\frac{d_a}{r_a}= \mu^{\max}_n$. If $a=1$, then $d_a=0$ so $\mu^{\max}_n=0$ while $\mu^{\max\prime}_n$ is by definition nonnegative. Otherwise, let $i=r_1+\cdots+r_{a-1}$, so $\frac{d_{a-1}}{r_{a-1}}\leq \mu^{\max}_i=\mu^{\max\prime}_i$ by definition and the assumption hypothesis. Then, $r_a = n-i$
\begin{align*}
    \frac{d_a}{r_a}\leq \mu^{\max\prime}_i+\frac{1}{k},
\end{align*}
so $d_a\leq \lfloor(\mu^{\max\prime}_i+\frac{1}{k})(n-i)\rfloor$. Then, by definition $\mu^{\max\prime}_n\geq \mu^{\max}_n$. Therefore, we are done and $\mu^{\max\prime}_n=\mu^{\max}_n$ for all $n$. 
\end{proof}

We define $B_k^{\up}(n)$ recursively by $B_k^{\up}(1) = 0$ and
\[ B_k^{\up}(n) = \max\{B_k^{\up}(i)+\lfloor(\mu^{\max}_i+\frac{1}{k})(n-i)\rfloor\mid 0<i<n\} .\]

For $B_k^{\low}$, we let $B_k^{\low}(1) = 0$ and let $i(n)$ be the smallest $i$ that maximizes $\frac{\lfloor(\mu^{\max}(i)+\frac{1}{k})(n-i)\rfloor}{n-i}$. Then define $B_k^{\low}$ inductively by
\[ B_k^{\low}(n)=B_k^{\low}(i(n))+\lfloor(\mu^{\max}(i(n))+\frac{1}{k})(n-i(n))\rfloor.  \]

We now show that $B_k(n)$ is bounded by $B_k^{\up}(n)$ and $B_k^{\low}(n)$.
\begin{lemma}
\label{rbounds}
We have $B_k^{\low}(n) \leq B_k(n) \leq B_k^{\up}(n)$.
\end{lemma}
\begin{proof}


First we show $B_k^{\low}(n) \leq B_k(n)$ by induction. To do this, we show by induction that $B_k^{\low}(n)$ is always achieved by an admissible sequence $(d_1,r_1),\ldots,(d_a,r_a)$ with $\frac{d_a}{r_a} = \mu^{\max}(n)$ and $d_1+\cdots+d_a = B_k^{\low}(n)$. The base case $n=1$ vacuous, so we assume $n>1$. Let $i$ be the minimal index maximizing $\frac{\lfloor(\mu^{\max}_i+\frac{1}{k})(n-i)\rfloor}{n-i}$. By the induction assumption, there is an admissible sequence $(d_1,r_1),\ldots,(d_{a-1},r_{a-1})$ achieving $\frac{d_{a-1}}{r_{a-1}} = \mu^{\max}_i$ and $d_1+\cdots+d_{a-1} = B_k^{\low}(i)$. By appending $(\lfloor(\mu^{\max}_i+\frac{1}{k})(n-i)\rfloor,n-i)$ to the sequence we get an admissible sequence $(d_1,r_1),\ldots,(d_a,r_a)$ with  $\frac{d_a}{r_a} = \mu^{\max}_n$ and $d_1+\cdots+d_a = B_k^{\low}(n)$. 

Finally, we show $B_k^{\up}(n) \geq B_k(n)$ by induction. The base case $n=1$ is vacuous, so assume $n>1$. Let $(d_1,r_1),\ldots(d_a,r_a)$ be an admissible sequence in $A_n$ achieving $d_1+\cdots+d_a=B_k(n)$. If $a=1$, then $B_k(n)=0$ and $B_k^{\up}(n)$ is always nonnegative by definition. If $a>1$, then let $i=r_1+\cdots+r_{a-1}$ so $(d_1,r_1),\ldots(d_{a-1},r_{a-1})$ is an admissible sequence in $A_i$. By the inductive hypothesis, $d_1+\cdots+d_{a-1}\leq B_k^{\up}(i)$. We have $r_a=n-i$ and the maximum $d_a$ can be is $\lfloor(\mu^{\max}(i)+\frac{1}{k})(n-i)\rfloor$. Therefore,
\begin{align*}
   d_1+\cdots+d_{a}\leq B_k^{\up}(i)+\lfloor(\mu^{\max}(i)+\frac{1}{k})(n-i)\rfloor\leq B_k^{\up}(n),
\end{align*}
finishing the proof.
\end{proof}

From experimental evidence, we suspect $B_k^{\low}(n)$ and $B_k^{\up}(n)$ always coincide, which would give a recursive algorithm for $B_k(n)$. However, to give results for small values of $k$ and all $n$, we want to have a finite criterion that can be verified by a computer. We believe admissible sequences achieving $B_k(n)$ will always following a periodic structure in $n$ with $k$ fixed reflected in \Cref{periodic} below.
\begin{lemma}
\label{periodic}
Suppose $\mu^{\max}(i_0)=\frac{k-1}{k}$ for some $i_0$. Then $\mu^{\max}(n) = \mu^{\max}(n-i_0) + 1$ all $n\geq i_0$. If in addition $B_k^{\up}(i)=B_k^{\low}(i)$ for each $i\leq 3i_0$, then $B_k^{\low}(n)=B_k(n)=B_k^{\up}(n)$ for all $n$. 
\end{lemma}

Using \Cref{periodic}, one can show $B_k(n+i_0) = B_k(n)+n+B_k(i_0)$. Iterating this shows
\begin{align*}
     B_k(n+Ni_0) = B_k(n)+nN+NB_k(i_0) + \frac{N(N-1)i_0}{2}
\end{align*}
for $1<n\leq i_0$ and $N\geq 0$. In particular, $B_k(n)=\Theta(\frac{1}{i_0}N^2)$. 

\begin{proof}
First note that if $\mu^{\max}(i)=m+\frac{k-1}{k}$ for $m$ an integer, then 
\begin{align}
\label{eq:k1k}
    \mu^{\max}(i+1)&=\max_j \{\frac{\lfloor(\mu^{\max}(j)+\frac{1}{k})(i+1-j)\rfloor}{i+1-j}\mid 0<j\leq i\}\\
    &= \lfloor m+\frac{k-1}{k}+\frac{1}{k}\rfloor = m+1.
\end{align}
In particular, there is a unique $i_0$ for which $\mu^{\max}(i_0)=\frac{k-1}{k}$ and $\mu^{\max}(i_0+1)=1=1+\mu^{\max}(1)$. 

We will now show $\mu^{\max}(n) = \mu^{\max}(n-i_0) + 1$ for all $n\geq i_0$ using induction on $n$. For the case $n=i_0+1$, $\mu^{\max}(i_0+1)=1$ from above. 

Now suppose $n>i_0+1$.  We first note that $$\mu^{\max}(i_0 \lfloor \frac{n-1}{i_0}\rfloor)=(\lfloor\frac{n-1}{i_0}\rfloor-1)+\frac{k-1}{k}$$ is $\frac{k-1}{k}$ by induction.

Now, we claim that $\mu^{\max}(i)<\mu^{\max}(i_0 \lfloor \frac{n-1}{i_0}\rfloor)$ for all $i<i_0 \lfloor \frac{n-1}{i_0}\rfloor$. Since $\mu^{\max}(j)$ is weakly increasing in $j$, the point is to prove they are not equal. If $\mu^{\max}(i)$ was equal to $\mu^{\max}(i_0 \lfloor \frac{n-1}{i_0}\rfloor)$, then $\mu^{\max}(i+1)= \mu^{\max}(i)$, which contradicts \eqref{eq:k1k}. 

By Lemma \ref{lem:muMaxComputation}, $\mu^{\max}(n)$ will be determined by the $i$ between $0$ and $n$ such that $\frac{\lfloor(\mu^{\max}(i)+\frac{1}{k})(n-i)\rfloor}{n-i}$ is maximized. Let $i(n)$ be this $i$. We claim $i(n) \geq i_0\lfloor \frac{n-1}{i_0}\rfloor$. To get a contradiction, suppose that $i(n) <i_0 \lfloor \frac{n-1}{i_0}\rfloor$.

Since we have shown above that $\mu^{\max}(i(n))<\mu^{\max}(i_0 \lfloor \frac{n-1}{i_0} \rfloor)$, $\mu^{\max}(i(n))+\frac{1}{k}< \mu^{\max}(i_0 \lfloor \frac{n-1}{i_0} \rfloor)+\frac{1}{k}=\lfloor \frac{n-1}{i_0}\rfloor$, contradicting $i(n) <i_0 \lfloor \frac{n-1}{i_0}\rfloor$.

Since $i(n) \geq i_0\lfloor \frac{n}{i_0}\rfloor$, 
\begin{align*}
   \mu^{\max}(n) &=  \frac{\lfloor(\mu^{\max}(i(n))+\frac{1}{k})(n-i(n))\rfloor}{n-i(n)}\\
   &= \frac{\lfloor(\mu^{\max}(i(n)-i_0)+1+\frac{1}{k})(n-i(n))\rfloor}{n-i(n)}\qquad\text{by induction}\\
    &= \frac{\lfloor(\mu^{\max}(i(n)-i_0)+\frac{1}{k})((n-i_0)-(i(n)-i_0))\rfloor}{(n-i_0)-(i(n)-i_0)}+1\\
    &= \max_j \left\{\frac{\lfloor(\mu^{\max}(j)+\frac{1}{k})(n-i_0-j)\rfloor}{n-i_0-j}\mid 0<j\leq n-i_0\right\}+1\qquad\text{by induction}\\
   &= \mu^{\max}(n-i_0)+1\qquad\text{by definition}
\end{align*}
This concludes our induction for $\mu^{\max}$. From our proof, we also see that
\begin{align}
\label{eq:i(n)}
    i(n)-i_0 &= i(n-i_0). 
\end{align}

Next, we want to show the statement regarding $B_k(n)$. It suffices to show that $B_k^{\low}(n)=B_k^{\up}(n)$ for all $n$. We will show this by induction and can assume $n> 3i_0$ and $B_k^{\low}(i)=B_k^{\up}(i)$ for all $0<i<n$. As before, let $i(n)$ be the minimum $i$ that maximizes $\frac{\lfloor(\mu^{\max}(i)+\frac{1}{k})(n-i)\rfloor}{n-i}$. 
By definition, we want to show
\begin{align}
\label{eq:Bkeq}
    B_k(i(n))+\lfloor(\mu^{\max}(i(n))+\frac{1}{k})(n-i(n))\rfloor&= \max\{B_k(i)+\lfloor(\mu^{\max}(i)+\frac{1}{k})(n-i)\rfloor\mid 0<i<n\}.
\end{align}
The inequality $\leq$ is clear as the left side is one of the terms on the right side. Let $i'$ be an index maximizing the right side. We want to show that $i'> i_0$. If $i'\leq i_0$, then $B^{\up}_k(i')+\lfloor(\mu^{\max}(i')+\frac{1}{k})(n-i')\rfloor$ is less than $n$ as $\mu^{\max}_{i'}+\frac{1}{k}\leq 1$ and $B^{\up}(i')\leq \mu^{\max}(i')i'<i'$. We also crudely bound $B_k^{\low}(n)$ from below. 

To do so, we first bound $B_k^{\low}(3i_0)$ by $3i_0$. From the statement of \Cref{periodic} regarding $\mu^{\max}$, we know $\mu^{\max}(j)\geq m$ for all $i>m\cdot i_0$. Then, 
\begin{align*}
    B_k^{\low}(n)&\geq 0\cdot i_0+1\cdot i_0+2\cdot i_0+3(n-3i_0).
\end{align*}
 Since $n>3i_0$, $3i_0 + 3(n-3i_0)=2(n-3i_0)+n>n$, yielding a contradiction.
 
 Since $i'>i_0$, the right side of \eqref{eq:Bkeq} is
 \begin{align*}
     \max\{B_k(i)+\lfloor(\mu^{\max}(i)+\frac{1}{k})(n-i)\rfloor\mid 0<i<n\}&=\\
     \max\{B_k(i)+\lfloor(\mu^{\max}(i)+\frac{1}{k})(n-i)\rfloor\mid i_0<i<n\}&=\\
     \max\{B_k(i+i_0)+\lfloor(\mu^{\max}(i+i_0)+\frac{1}{k})(n-i-i_0)\rfloor\mid 0<i<n-i_0\}&=\\
    \max\{B_k(i)+i+B_k(i_0)+\lfloor(\mu^{\max}(i)+1+\frac{1}{k})(n-i-i_0)\rfloor\mid 0<i<n-i_0\}&=\\
    \max\{B_k(i)+\lfloor(\mu^{\max}(i)+\frac{1}{k})(n-i-i_0)\rfloor\mid 0<i<n-i_0\}+n-i_0+B_k(i_0)&.
 \end{align*}
 But $\max\{B_k(i)+\lfloor(\mu^{\max}(i)+\frac{1}{k})(n-i-i_0)\rfloor\mid 0<i<n-i_0\}=B(n-i_0)$ by induction. Looking at the left side of \eqref{eq:Bkeq}, we get
 \begin{align*}
     B_k(i(n))+\lfloor(\mu^{\max}(i(n))+\frac{1}{k})(n-i(n))\rfloor&=\\
     B_k(i(n)-i_0)+(i(n)-i_0)+B_k(i_0)+\lfloor(\mu^{\max}(i(n)-i_0+i_0)+\frac{1}{k})(n-i(n))\rfloor&=\\
     B_k(i(n)-i_0)+(i(n)-i_0)+B_k(i_0)+\lfloor(\mu^{\max}(i(n)-i_0)+1+\frac{1}{k})(n-i(n))\rfloor&=\\
    B_k(i(n)-i_0)+\lfloor(\mu^{\max}(i(n)-i_0)+\frac{1}{k})(n-i_0-(i(n)-i_0)\rfloor+n-i_0+B_k(i_0)&=\\
    B_k(n-i_0)+B_k(i_0)+n-i_0&.\qquad\text{by \eqref{eq:i(n)}}
 \end{align*}
 Therefore, both sides of \eqref{eq:Bkeq} are equal, which is what we wanted. 
\end{proof}

We can verify the conditions of \Cref{periodic} using a Python program for small $k$. For example, the answer for $k=2,3,4,5$ are given below. 
\begin{corollary}
\label{cor:manyquadrics}
We have the following closed-form expressions for $B_k(n)$ for $k=2,3,4,5$. For $k=2$ and $n\geq 0$
\begin{align*}
    B_2(3n+1)&=\frac{3n^2+n}{2} \quad
    B_2(3n+2)=\frac{3n^2+3n}{2} \quad
    B_2(3n+3)=\frac{3n^2+5n+2}{2}.
\end{align*}
For $k=3$,
\begin{align*}
    B_3(7n+1)&=\frac{7n^2+n}{2}\quad
    B_3(7n+2)=\frac{7n^2+3n}{2}\quad
    B_3(7n+3)=\frac{7n^2+5n}{2}\\
    B_3(7n+4)&=\frac{7n^2+7n+2}{2}\quad
    B_3(7n+5)=\frac{7n^2+9n+2}{2}\quad
    B_3(7n+6)=\frac{7n^2+11n+4}{2}\\
    B_3(7n+7)&=\frac{7n^2+13n+6}{2}.
\end{align*}
For $k=4$,
\begin{align*}
    B_4(11n+1)&=\frac{11n^2+n}{2}\quad
    B_4(11n+2)=\frac{11n^2+3n}{2}\quad
    B_4(11n+3)=\frac{11n^2+5n}{2}\\
    B_4(11n+4)&=\frac{11n^2+7n}{2}\quad
    B_4(11n+5)=\frac{11n^2+9n+2}{2}\quad
    B_4(11n+6)=\frac{11n^2+11n+2}{2}\\
    B_4(11n+7)&=\frac{11n^2+13n+4}{2}\quad
    B_4(11n+8)=\frac{11n^2+15n+4}{2}\quad
    B_4(11n+9)=\frac{11n^2+17n+6}{2}\\
    B_4(11n+10)&=\frac{11n^2+19n+8}{2}\quad
    B_4(11n+11)=\frac{11n^2+21n+10}{2}.\quad
\end{align*}
For $k=5$,
\begin{align*}
    B_5(19n+1)&=\frac{19n^2+n}{2}\quad
    B_5(19n+2)=\frac{19n^2+3n}{2}\quad
    B_5(19n+3)=\frac{19n^2+5n}{2}\\
    B_5(19n+4)&=\frac{19n^2+7n}{2}\quad
    B_5(19n+5)=\frac{19n^2+9n}{2}\quad
    B_5(19n+6)=\frac{19n^2+11n+2}{2}\\
    B_5(19n+7)&=\frac{19n^2+13n+2}{2}\quad
    B_5(19n+8)=\frac{19n^2+15n+2}{2}\quad
    B_5(19n+9)=\frac{19n^2+17n+4}{2}\\
    B_5(19n+10)&=\frac{19n^2+19n+4}{2}\quad
    B_5(19n+11)=\frac{19n^2+21n+6}{2}\quad
    B_5(19n+12)=\frac{19n^2+23n+6}{2}\\
    B_5(19n+13)&=\frac{19n^2+25n+8}{2}\quad
    B_5(19n+14)=\frac{19n^2+27n+10}{2}\quad
    B_5(19n+15)=\frac{19n^2+29n+10}{2}\\
    B_5(19n+16)&=\frac{19n^2+31n+12}{2}\quad
    B_5(19n+17)=\frac{19n^2+33n+14}{2}\quad
    B_5(19n+18)=\frac{19n^2+35n+16}{2}\\
    B_5(19n+19)&=\frac{19n^2+37n+18}{2}.
\end{align*}
\end{corollary}

\bibliographystyle{alpha} \bibliography{references.bib}

\end{document}